\documentclass[12pt, a4paper]{article}

\usepackage{amsmath, amssymb, amsthm}
\usepackage{array}
\usepackage{booktabs}
\usepackage{geometry}
\usepackage{hyperref}
\usepackage{caption}
\hypersetup{colorlinks=true, linkcolor=blue, urlcolor=blue}
\usepackage{tikz}
\usepackage{float}
\usetikzlibrary{positioning, calc}

\usepackage{geometry}
\usepackage{makecell}

\newtheorem{definition}{Definition}[section]
\newtheorem{example}{Example}[section]
\newtheorem{remark}{Remark}[section]
\newtheorem{theorem}{Theorem}[section]
\newtheorem{proposition}{Proposition}[section]
\newtheorem{corollary}{Corollary}[section]
\newtheorem{problem}{Problem}[section]

\title{\textbf{On the Word-Representability of Tensor Product Graphs}}
\author{Nawaf Shafi Alshammari\footnote{Department of Mathematics, Faculty of Science, University of Hafr Al Batin, Hafr Al Batin 31991, Saudi Arabia. {\bf Email:} nsalshammari@uhb.edu.sa} \footnote{Department of Mathematics and Statistics, University of Strathclyde, 26 Richmond Street, Glasgow G1, 1XH, United Kingdom. {\bf Email:} nawaf.alshammari@strath.ac.uk}}

\begin{document}
\maketitle

\begin{abstract}
Word-representable graphs are a class of graphs that can be represented by words, where edges and non-edges are determined by the alternation of letters in those words. The tensor product \(G \times H\) (also known as the direct product or Kronecker product) is one of the four standard graph products. Problem 7.2.5 in Kitaev and Lozin's book \emph{Words and Graphs} (Springer, 2015) raised three open questions regarding the word-representability of tensor products.

Despite extensive research on word-representable graphs, the word-representability of tensor products seems to have received no attention. This paper not only answers 2.5 of the questions posed by Kitaev and Lozin in Problem 7.2.5, but also initiates a systematic study, focusing on four fundamental families: wheel graphs \(W_n\), complete graphs \(K_n\), the Mycielskian of the cycle graph \(\mu_n\), and the extended Mycielskian of the cycle graph \(\mu'_n\).

Among our main results, we prove that \(W_{2n} \times G\), \(\mu_{2n} \times G\),  and \(\mu'_{2n} \times G\) are always word-representable for any graph \(G\); that \(K_n \times K_m\) is word-representable if and only if \(\min\{n,m\} \leq 3\); and that tensor products \(G \times H\) containing \(W_{2n+1} \times W_{2m+1}\) or \(\mu_{2n+1} \times \mu_{2m+1}\) or \(\mu'_{2n+1} \times \mu'_{2m+1}\)  as induced subgraphs are non-word-representable. Our proofs exploit the hereditary nature of non-word-representability and the presence of non-comparability neighbourhoods.\\

\end{abstract}

\section{Introduction}

Two letters $x$ and $y$ alternate in a word $w$ if, after deleting
from $w$ all letters except the copies of $x$ and $y$, we obtain either a word
$xyxy\cdots$ or a word $yxyx\cdots$ (of even or odd length). A graph
$G=(V,E)$ is \emph{word-representable} if there exists a word $w$ over
the alphabet $V$ such that letters $x$ and $y$, $x \neq y$, alternate in
$w$ if and only if $xy \in E$; each letter in $V$ must appear in $w$. The
unique minimal (by the number of vertices) non-word-representable graph
on 6 vertices is the wheel graph $W_5$, presented in Figure~\ref{fig:wheels}
(which is the cycle graph $C_5$ with an all-adjacent vertex added), while there are 25
non-word-representable graphs on 7 vertices~\cite{kitaev2015}.

The introduction of word-representable graphs was influenced by
alternation word digraphs, which served as a tool to study the celebrated
Perkins semigroup, a central concept in semigroup theory since 1960,
particularly as a source of examples and counterexamples.

An orientation of a graph is \emph{semi-transitive} if it is acyclic
(there are no directed cycles), and for any directed path
$v_0\to v_1\to\cdots\to v_k$, either there is no edge between $v_0$ and
$v_k$, or $v_i\to v_j$ is an edge for all $0\le i<j\le k$. An undirected
graph is \emph{semi-transitive} if it admits a semi-transitive
orientation. A \emph{shortcut} in an acyclically oriented graph $G$ is
an induced non-transitive subgraph $G'$ on vertices
$\{v_0,v_1,\ldots,v_k\}$, $k\ge 3$, containing
$v_0\to v_1\to\cdots\to v_k$ and $v_0\to v_k$ (that is, $v_i\to v_j$
is not present in $G'$ for at least one pair of $i$ and $j$ with
$0\le i<j\le k$). The edge $v_0\to v_k$ is called the
\emph{shortcutting edge}.

A fundamental result in the area of word-representable graphs is the
following theorem.

\begin{theorem}[\cite{halldorsson2016}]
\label{thm:kitaev}
A graph $G$ is word-representable if and only if it admits a
semi-transitive orientation, that is, an acyclic and shortcut-free
orientation of its edges.
\end{theorem}

The following corollary is used multiple times throughout the paper.

\begin{corollary}[\cite{halldorsson2016}]
\label{cor:3col}
Any $3$-colourable graph is word-representable.
\end{corollary}

\begin{definition}[Comparability graph]
A graph $G$ is called a \emph{comparability graph} if its edges can be oriented so that the resulting directed graph is transitive, that is, whenever $(u,v)$ and $(v,w)$ are directed edges, then $(u,w)$ is also a directed edge. A graph that is not a comparability graph is called a \emph{non-comparability graph}.
\end{definition}

\begin{definition}[Non-comparability neighbourhood]
Let $G=(V,E)$ be a graph and let $v \in V$. The \emph{neighbourhood} of $v$, denoted by $N(v)$, is the set of all vertices adjacent to $v$. We say that $v$ has a \emph{non-comparability neighbourhood} if the induced subgraph $G[N(v)]$ is a non-comparability graph.
\end{definition}

We will use the following theorem.

\begin{theorem}[\cite{kitaev2008}]\label{non-comp-neighb}
If a graph $G$ contains a vertex with a non-comparability neighbourhood, then $G$ is not word-representable.
\end{theorem}

\subsection{Graphs considered in this paper and their labeling}

\begin{definition}[Complete graph]
The \emph{complete graph} $K_n$ is the simple graph on $n$ vertices
in which every pair of distinct vertices is adjacent. 
\end{definition}

\begin{remark}
It is well known that $K_n$ is word-representable, as it can be represented by any permutation of its vertices.
\end{remark}

%
%
%

\begin{definition}[Wheel graph]

The \emph{wheel graph} $W_n$ is the graph obtained by taking a cycle
$C_n$ on $n$ vertices and adding a single \emph{central vertex}
connected to every vertex of $C_n$. The vertices of the cycle are
called the \emph{outer vertices}. 
\end{definition}

\begin{theorem}[\cite{kitaev2008}]
\label{thm:w5}
If a graph $G$ contains the wheel graph $W_{2k+1}$ for any $k \geq 2$
as an induced subgraph, then $G$ is non-word-representable.
\end{theorem}

\begin{remark}
Note that $W_3 \cong K_4$, and hence $W_3$ is word-representable. For $n \geq 4$, the chromatic number of $W_n$ is $3$ if $n$ is even, and $4$ if $n$ is odd. Consequently, by Theorem~\ref{thm:w5} and Corollary~\ref{cor:3col}, $W_n$ is word-representable if and only if $n$ is even or $n = 3$.
\end{remark}

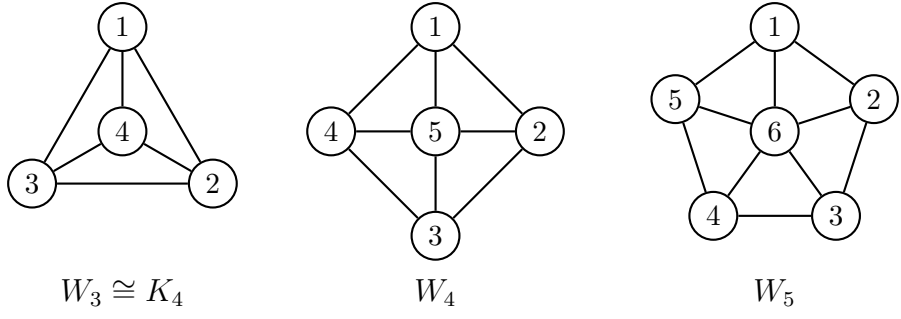
\begin{figure}[!t]
\centering
\begin{tikzpicture}[
    V/.style={circle, draw, thick, fill=white,
              minimum size=18pt, inner sep=0pt, font=\small},
    scale=1.15
]

\node[V] (w3h) at (0,0)              {$4$};
\node[V] (w3a) at ($(90:1.2)$)       {$1$};
\node[V] (w3b) at ($(330:1.2)$)      {$2$};
\node[V] (w3c) at ($(210:1.2)$)      {$3$};
\draw[thick] (w3a)--(w3b)--(w3c)--(w3a);
\draw[thick] (w3h)--(w3a) (w3h)--(w3b) (w3h)--(w3c);
\node at (0,-1.85) {$W_3\cong K_4$};

\node[V] (w4h) at (3.6,0)                    {$5$};
\node[V] (w4a) at ($(90:1.2)  +(3.6,0)$)    {$1$};
\node[V] (w4b) at ($(0:1.2)   +(3.6,0)$)    {$2$};
\node[V] (w4c) at ($(270:1.2) +(3.6,0)$)    {$3$};
\node[V] (w4d) at ($(180:1.2) +(3.6,0)$)    {$4$};
\draw[thick] (w4a)--(w4b)--(w4c)--(w4d)--(w4a);
\draw[thick] (w4h)--(w4a) (w4h)--(w4b) (w4h)--(w4c) (w4h)--(w4d);
\node at (3.6,-1.85) {$W_4$};

\node[V] (w5h) at (7.5,0)                    {$6$};
\node[V] (w5a) at ($(90:1.2)  +(7.5,0)$)    {$1$};
\node[V] (w5b) at ($(18:1.2)  +(7.5,0)$)    {$2$};
\node[V] (w5c) at ($(306:1.2) +(7.5,0)$)    {$3$};
\node[V] (w5d) at ($(234:1.2) +(7.5,0)$)    {$4$};
\node[V] (w5e) at ($(162:1.2) +(7.5,0)$)    {$5$};
\draw[thick] (w5a)--(w5b)--(w5c)--(w5d)--(w5e)--(w5a);
\draw[thick] (w5h)--(w5a) (w5h)--(w5b) (w5h)--(w5c) (w5h)--(w5d) (w5h)--(w5e);
\node at (7.5,-1.85) {$W_5$};

\end{tikzpicture}
\caption{The wheel graphs $W_3$, $W_4$, and $W_5$.
Graphs $W_3$ and $W_4$ are word-representable, whereas $W_5$ is not.
The largest label always corresponds to the hub vertex.}
\label{fig:wheels}
\end{figure}

\begin{definition}[Mycielski of cycle graph]
The \emph{Mycielski of cycle graph} $\mu_n = \mu(C_n)$ is obtained by
applying one Mycielski construction step to the cycle $C_n$. Given
$C_n$ with vertices $v_1, v_2, \ldots, v_n$, the graph $\mu_n$ is
constructed as follows:
\begin{itemize}
    \item Keep all vertices and edges of $C_n$ as the
          \emph{top row}.
    \item Add a \emph{shadow vertex} $u_i$ for each vertex $v_i$,
          where $u_i$ is connected to every neighbour of $v_i$ in
          $C_n$. These form the \emph{middle row}.
    \item Add a single \emph{root vertex} $w$ connected to all
          shadow vertices $u_1, u_2, \ldots, u_n$.
\end{itemize}
Thus $\mu_n = \mu(C_n)$ has $2n + 1$ vertices arranged in three layers:
$n$ top vertices, $n$ shadow vertices, and $1$ root vertex.
\end{definition}


\begin{remark}\label{odd-M-non-wr} 
The graph $\mu_3$ is isomorphic to Graph $13$ in the list of $25$ non-word-representable graphs in \cite{KS}. Kitaev and Pyatkin~\cite{KP} proved that $\mu_5$ is not word-representable. However, a systematic study of the word-representability of Mycielski graphs of cycles was initiated by Hameed in~\cite{Hameed2024}, where it was shown that $\mu_7$, $\mu_9$, and $\mu_{11}$ are not word-representable, and it was also conjectured that the Mycielski graph of a graph $G$ is word-representable if and only if $G$ is bipartite. The conjecture was confirmed in~\cite{kitaev2025} by Kitaev and Pyatkin, leading, in particular, to the fact that $\mu_{2n+1}$ is not word-representable for $n \geq 1$.
\end{remark}

\begin{remark}
In this paper, we label $\mu_{n}$ according to the pattern shown in Figure~\ref{fig:mycielski}.
\end{remark}

\begin{definition}[Extended Mycielski of cycle graph]
The \emph{extended Mycielski of cycle graph}
$\mu'_n = \mu'(C_n)$ is obtained by applying one extended Mycielski
construction step to the cycle $C_n$.
Given $C_n$ with vertices $v_1, v_2, \ldots, v_n$, the graph $\mu'_n$
is constructed as follows:
\begin{itemize}
  \item Keep all vertices and edges of $C_n$ as the top row.
  \item Add a shadow vertex $u_i$ for each vertex $v_i$, where $u_i$ is
        connected to every vertex of $C_n$ except $v_i$ itself.
        These form the middle row.
  \item Add a single root vertex $w$ connected to all shadow vertices
        $u_1, u_2, \ldots, u_n$.
\end{itemize}
Thus $\mu'_n = \mu'(C_n)$ has $2n+1$ vertices arranged in three layers:
$n$ top vertices, $n$ shadow vertices, and $1$ root vertex. The key
difference from $\mu_n$ is that each shadow vertex $u_i$ is adjacent to
all top-row vertices except its own corresponding vertex $v_i$, rather
than only to the neighbours of $v_i$ in $C_n$. See Figure~\ref{fig:mycielski-ext} for examples.
\end{definition}

\begin{remark}
Hameed~[2] proved a complete characterisation of semi-transitive extended
Mycielski graphs: $\mu'(G)$ is semi-transitive if and only if $G$ is a
bipartite graph. Since $C_{2n}$ is bipartite and $C_{2n+1}$ is not, it
follows immediately that $\mu'_{2n}$ is word-representable for all
$n \geq 2$, whereas $\mu'_{2n+1}$ is not word-representable for all
$n \geq 1$. This is precisely the same characterisation as for $\mu_n$
(see Remark~1.3). Consequently, all results in this paper concerning
$\mu_{2n}$ and $\mu_{2n+1}$ hold verbatim for $\mu'_{2n}$ and
$\mu'_{2n+1}$ respectively, with identical proofs. We will therefore skip
the proofs for $\mu'$ and simply state the results in the summary table.
\end{remark}

\begin{figure}[!t]
\centering
\begin{tikzpicture}[
    V/.style={circle, draw, thick, fill=white,
              minimum size=20pt, inner sep=0pt, font=\scriptsize},
    scale=1.0
]

\node[V] (a1) at (0.0,  0)    {$1$};
\node[V] (a2) at (1.0,  0)    {$2$};
\node[V] (a3) at (2.0,  0)    {$3$};
\node[V] (a4) at (0.0, -2.0)  {$4$};
\node[V] (a5) at (1.0, -2.0)  {$5$};
\node[V] (a6) at (2.0, -2.0)  {$6$};
\node[V] (a7) at (1.0, -3.2)  {$7$};
\draw[thick] (a1)--(a2)--(a3);
\draw[thick, bend left=50] (a1) to (a3);
\draw[thick] (a4)--(a2)  (a4)--(a3);
\draw[thick] (a5)--(a1)  (a5)--(a3);
\draw[thick] (a6)--(a1)  (a6)--(a2);
\draw[thick] (a7)--(a4) (a7)--(a5) (a7)--(a6);
\node at (1.0,-4) {$\mu_3$};

\node[V] (b1) at (3.5,  0)    {$1$};
\node[V] (b2) at (4.5,  0)    {$2$};
\node[V] (b3) at (5.5,  0)    {$3$};
\node[V] (b4) at (6.5,  0)    {$4$};
\node[V] (b5) at (3.5, -2.0)  {$5$};
\node[V] (b6) at (4.5, -2.0)  {$6$};
\node[V] (b7) at (5.5, -2.0)  {$7$};
\node[V] (b8) at (6.5, -2.0)  {$8$};
\node[V] (b9) at (5.0, -3.2)  {$9$};
\draw[thick] (b1)--(b2)--(b3)--(b4);
\draw[thick, bend left=40] (b1) to (b4);
\draw[thick] (b5)--(b2)  (b5)--(b4);
\draw[thick] (b6)--(b1)  (b6)--(b3);
\draw[thick] (b7)--(b2)  (b7)--(b4);
\draw[thick] (b8)--(b3)  (b8)--(b1);
\draw[thick] (b9)--(b5) (b9)--(b6) (b9)--(b7) (b9)--(b8);
\node at (5.0,-4) {$\mu_4$};

\node[V] (c1) at (7.9,  0)    {$1$};
\node[V] (c2) at (8.9,  0)    {$2$};
\node[V] (c3) at (9.9,  0)    {$3$};
\node[V] (c4) at (10.9, 0)    {$4$};
\node[V] (c5) at (11.9, 0)    {$5$};
\node[V] (c6)  at (7.9,  -2.0) {$6$};
\node[V] (c7)  at (8.9,  -2.0) {$7$};
\node[V] (c8)  at (9.9,  -2.0) {$8$};
\node[V] (c9)  at (10.9, -2.0) {$9$};
\node[V] (c10) at (11.9, -2.0) {$10$};
\node[V] (c11) at (9.9, -3.2)  {$11$};
\draw[thick] (c1)--(c2)--(c3)--(c4)--(c5);
\draw[thick, bend left=35] (c1) to (c5);
\draw[thick] (c6)--(c2)  (c6)--(c5);
\draw[thick] (c7)--(c1)  (c7)--(c3);
\draw[thick] (c8)--(c2)  (c8)--(c4);
\draw[thick] (c9)--(c3)  (c9)--(c5);
\draw[thick] (c10)--(c4) (c10)--(c1);
\draw[thick] (c11)--(c6) (c11)--(c7) (c11)--(c8) (c11)--(c9) (c11)--(c10);
\node at (9.9,-4) {$\mu_5$};

\end{tikzpicture}
\caption{The non-word-representable Mycielski of cycle graphs $\mu_3$, $\mu_4$, and $\mu_5$.
Each graph consists of a top row (the original cycle $C_n$), a middle row (shadow vertices), and a single root vertex at the bottom, which receives the largest label.
We follow a consistent labeling pattern for all $\mu_n$.}
\label{fig:mycielski}
\end{figure}
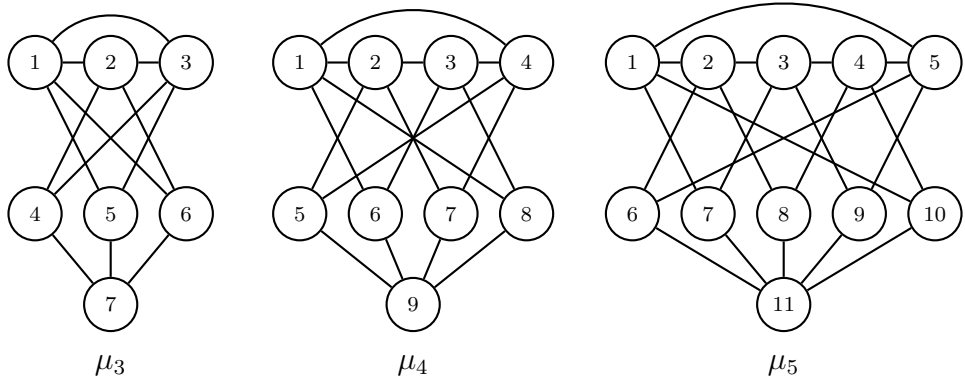

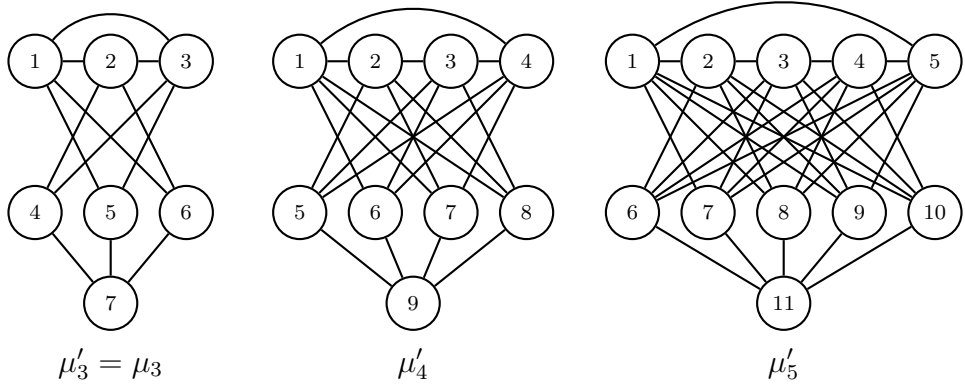
\begin{figure}[!t]
\centering
\begin{tikzpicture}[
    V/.style={circle, draw, thick, fill=white,
              minimum size=20pt, inner sep=0pt, font=\scriptsize},
    scale=1.0
]

\node[V] (a1) at (0.0,  0)    {$1$};
\node[V] (a2) at (1.0,  0)    {$2$};
\node[V] (a3) at (2.0,  0)    {$3$};
\node[V] (a4) at (0.0, -2.0)  {$4$};
\node[V] (a5) at (1.0, -2.0)  {$5$};
\node[V] (a6) at (2.0, -2.0)  {$6$};
\node[V] (a7) at (1.0, -3.2)  {$7$};
\draw[thick] (a1)--(a2)--(a3);
\draw[thick, bend left=50] (a1) to (a3);
\draw[thick] (a4)--(a2)  (a4)--(a3);
\draw[thick] (a5)--(a1)  (a5)--(a3);
\draw[thick] (a6)--(a1)  (a6)--(a2);
\draw[thick] (a7)--(a4) (a7)--(a5) (a7)--(a6);
\node at (1.0,-4) {$\mu'_3=\mu_3$};

\node[V] (b1) at (3.5,  0)    {$1$};
\node[V] (b2) at (4.5,  0)    {$2$};
\node[V] (b3) at (5.5,  0)    {$3$};
\node[V] (b4) at (6.5,  0)    {$4$};
\node[V] (b5) at (3.5, -2.0)  {$5$};
\node[V] (b6) at (4.5, -2.0)  {$6$};
\node[V] (b7) at (5.5, -2.0)  {$7$};
\node[V] (b8) at (6.5, -2.0)  {$8$};
\node[V] (b9) at (5.0, -3.2)  {$9$};
\draw[thick] (b1)--(b2)--(b3)--(b4);
\draw[thick, bend left=40] (b1) to (b4);
\draw[thick] (b5)--(b2) (b5)--(b3) (b5)--(b4);
\draw[thick] (b6)--(b1)  (b6)--(b4) (b6)--(b3);
\draw[thick] (b7)--(b2)  (b7)--(b1) (b7)--(b4);
\draw[thick] (b8)--(b3)  (b8)--(b2) (b8)--(b1);
\draw[thick] (b9)--(b5) (b9)--(b6) (b9)--(b7) (b9)--(b8);
\node at (5.0,-4) {$\mu'_4$};

\node[V] (c1) at (7.9,  0)    {$1$};
\node[V] (c2) at (8.9,  0)    {$2$};
\node[V] (c3) at (9.9,  0)    {$3$};
\node[V] (c4) at (10.9, 0)    {$4$};
\node[V] (c5) at (11.9, 0)    {$5$};
\node[V] (c6)  at (7.9,  -2.0) {$6$};
\node[V] (c7)  at (8.9,  -2.0) {$7$};
\node[V] (c8)  at (9.9,  -2.0) {$8$};
\node[V] (c9)  at (10.9, -2.0) {$9$};
\node[V] (c10) at (11.9, -2.0) {$10$};
\node[V] (c11) at (9.9, -3.2)  {$11$};
\draw[thick] (c1)--(c2)--(c3)--(c4)--(c5);
\draw[thick, bend left=35] (c1) to (c5);
\draw[thick] (c6)--(c2) (c6)--(c3) (c6)--(c4) (c6)--(c5);
\draw[thick] (c7)--(c1)  (c7)--(c4) (c7)--(c5) (c7)--(c3);
\draw[thick] (c8)--(c1) (c8)--(c2) (c8)--(c5)  (c8)--(c4);
\draw[thick] (c9)--(c1) (c9)--(c2) (c9)--(c3)  (c9)--(c5);
\draw[thick] (c10)--(c2) (c10)--(c3) (c10)--(c4) (c10)--(c1);
\draw[thick] (c11)--(c6) (c11)--(c7) (c11)--(c8) (c11)--(c9) (c11)--(c10);
\node at (9.9,-4) {$\mu'_5$};

\end{tikzpicture}
\caption{The non-word-representable extended Mycielski of cycle graphs $\mu'_3$, $\mu'_4$, and $\mu'_5$.
Each graph consists of a top row (the original cycle $C_n$), a middle row (shadow vertices), and a single root vertex at the bottom, which receives the largest label.}
\label{fig:mycielski-ext}
\end{figure}

\subsection{Results of the paper} 

Despite the extensive attention given to word-representable graphs in the literature, to the best of our knowledge, no results are known on the word-representability of tensor graph products. In \cite{kitaev2015}, Kitaev and Lozin posed the following problem.

\begin{problem}[Problem 7.2.5 in \cite{kitaev2015}] Can anything definite be said about taking the tensor product of two word-representable graphs? Or the tensor product of a word-representable and a non-word-representable graph? Or the tensor product of two non-word-representable graphs?
\end{problem}

It follows from the results in this paper that the tensor product of two word-representable graphs, as well as that of a word-representable graph and a non-word-representable graph, may or may not be word-representable. We also provide examples showing that the tensor product of two non-word-representable graphs can be non-word-representable. However, it remains an open question whether the tensor product of two non-word-representable graphs can ever be word-representable.

Motivated by these observations, our goal is to initiate a systematic study of the word-representability of tensor products of graphs containing $W_n$, $K_n$, $\mu_n$, or $\mu'_n$ as induced subgraphs. In Section~\ref{tensor-sec}, we define the notion of the tensor product of graphs and present the properties relevant to our work. In Section~\ref{wr-tensor-prod-wheel}, we present results on tensor products involving wheel graphs $W_n$. The remaining cases are discussed in Section~\ref{remaining-sec}. Concluding remarks are provided in Section~\ref{conclusion-sec}. Finally, our key results are summarized in Table~\ref{tab-results}.

\begin{table}[!t]
\centering
\begin{tabular}{|c|c|c|c|c|c|c|c|c|}
\hline
\makecell{$G \setminus H$}
  & \makecell{$W_{2m}$ \\ \scriptsize$(m\geq 2)$}
  & \makecell{$W_{2m+1}$ \\ \scriptsize$(m\geq 2)$}
  & \makecell{$K_{m}$ \\ \scriptsize$(m\leq 3)$}
  & \makecell{$K_{m}$ \\ \scriptsize$(m\geq 4)$}
  & \makecell{$\mu_{2m}$ \\ \scriptsize$(m\geq 2)$}
  & \makecell{$\mu_{2m+1}$ \\ \scriptsize$(m\geq 1)$}
  & \makecell{$\mu'_{2m}$ \\ \scriptsize$(m\geq 2)$}
  & \makecell{$\mu'_{2m+1}$ \\ \scriptsize$(m\geq 1)$} \\
\hline
\makecell{$W_{2n}$ \\ \scriptsize$(n\geq 2)$}
  & \makecell{\textbf{WR} \\ \scriptsize(Thm~\ref{thm:w2n_wr})}
  & \makecell{\textbf{WR} \\ \scriptsize(Thm~\ref{thm:w2n_wr})}
  & \makecell{\textbf{WR} \\ \scriptsize(Thm~\ref{thm:w2n_wr})}
  & \makecell{\textbf{WR} \\ \scriptsize(Thm~\ref{thm:w2n_wr})}
  & \makecell{\textbf{WR} \\ \scriptsize(Thm~\ref{thm:w2n_wr})}
  & \makecell{\textbf{WR} \\ \scriptsize(Thm~\ref{thm:w2n_wr})}
  & \makecell{\textbf{WR} \\ \scriptsize(Thm~\ref{thm:w2n_wr})}
  & \makecell{\textbf{WR} \\ \scriptsize(Thm~\ref{thm:w2n_wr})} \\
\hline
\makecell{$W_{2n+1}$ \\ \scriptsize$(n\geq 2)$}
  & \makecell{\textbf{WR} \\ \scriptsize(Thm~\ref{thm:w2n_wr})}
  & \makecell{\textbf{NWR} \\ \scriptsize(Thm~\ref{thm:wheel})}
  & \makecell{\textbf{WR} \\ \scriptsize(Cor.~\ref{cor:3col})}
  & \makecell{\textbf{NWR} \\ \scriptsize(Thm~\ref{thm:w3-1})}
  & \makecell{\textbf{WR} \\ \scriptsize(Thm~\ref{thm:m2n_wr})}
  & \makecell{\textbf{Open} \\ \scriptsize(Prob.~\ref{prob1})}
  & \makecell{\textbf{WR} \\ \scriptsize(Thm~\ref{thm:m2n_wr})}
  & \makecell{\textbf{Open} \\ \scriptsize(Prob.~\ref{prob1})} \\
\hline
\makecell{$K_{n}$ \\ \scriptsize$(n\leq 3)$}
  & \makecell{\textbf{WR} \\ \scriptsize(Thm~\ref{thm:w2n_wr})}
  & \makecell{\textbf{WR} \\ \scriptsize(Cor.~\ref{cor:3col})}
  & \makecell{\textbf{WR} \\ \scriptsize(Cor.~\ref{cor:3col})}
  & \makecell{\textbf{WR} \\ \scriptsize(Cor.~\ref{cor:3col})}
  & \makecell{\textbf{WR} \\ \scriptsize(Cor.~\ref{cor:3col})}
  & \makecell{\textbf{WR} \\ \scriptsize(Cor.~\ref{cor:3col})}
  & \makecell{\textbf{WR} \\ \scriptsize(Cor.~\ref{cor:3col})}
  & \makecell{\textbf{WR} \\ \scriptsize(Cor.~\ref{cor:3col})} \\
\hline
\makecell{$K_{n}$ \\ \scriptsize$(n\geq 4)$}
  & \makecell{\textbf{WR} \\ \scriptsize(Thm~\ref{thm:w2n_wr})}
  & \makecell{\textbf{NWR} \\ \scriptsize(Thm~\ref{thm:w3-1})}
  & \makecell{\textbf{WR} \\ \scriptsize(Cor.~\ref{cor:3col})}
  & \makecell{\textbf{NWR} \\ \scriptsize(Thm~\ref{thm:kn_km})}
  & \makecell{\textbf{WR} \\ \scriptsize(Thm~\ref{thm:m2n_wr})}
  & \makecell{\textbf{NWR} \\ \scriptsize(Thm~\ref{thm:mu2n1_k4})}
  & \makecell{\textbf{WR} \\ \scriptsize(Thm~\ref{thm:m2n_wr})}
  & \makecell{\textbf{NWR} \\ \scriptsize(Thm~\ref{thm:mu2n1_k4})} \\
\hline
\makecell{$\mu_{2n}$ \\ \scriptsize$(n\geq 2)$}
  & \makecell{\textbf{WR} \\ \scriptsize(Thm~\ref{thm:w2n_wr})}
  & \makecell{\textbf{WR} \\ \scriptsize(Thm~\ref{thm:m2n_wr})}
  & \makecell{\textbf{WR} \\ \scriptsize(Thm~\ref{thm:m2n_wr})}
  & \makecell{\textbf{WR} \\ \scriptsize(Thm~\ref{thm:m2n_wr})}
  & \makecell{\textbf{WR} \\ \scriptsize(Thm~\ref{thm:m2n_wr})}
  & \makecell{\textbf{WR} \\ \scriptsize(Thm~\ref{thm:m2n_wr})}
  & \makecell{\textbf{WR} \\ \scriptsize(Thm~\ref{thm:m2n_wr})}
  & \makecell{\textbf{WR} \\ \scriptsize(Thm~\ref{thm:m2n_wr})} \\
\hline
\makecell{$\mu_{2n+1}$ \\ \scriptsize$(n\geq 1)$}
  & \makecell{\textbf{WR} \\ \scriptsize(Thm~\ref{thm:w2n_wr})}
  & \makecell{\textbf{Open} \\ \scriptsize(Prob.~\ref{prob1})}
  & \makecell{\textbf{WR} \\ \scriptsize(Cor.~\ref{cor:3col})}
  & \makecell{\textbf{NWR} \\ \scriptsize(Thm~\ref{thm:mu2n1_k4})}
  & \makecell{\textbf{WR} \\ \scriptsize(Thm~\ref{thm:m2n_wr})}
  & \makecell{\textbf{NWR} \\ \scriptsize(Thm~\ref{thm:mu2n1_mu2m1})}
  & \makecell{\textbf{WR} \\ \scriptsize(Thm~\ref{thm:m2n_wr})}
  & \makecell{\textbf{NWR} \\ \scriptsize(if $n=m$)}  \\
\hline
\makecell{$\mu'_{2n}$ \\ \scriptsize$(n\geq 2)$}
  & \makecell{\textbf{WR} \\ \scriptsize(Thm~\ref{thm:w2n_wr})}
  & \makecell{\textbf{WR} \\ \scriptsize(Thm~\ref{thm:m2n_wr})}
  & \makecell{\textbf{WR} \\ \scriptsize(Thm~\ref{thm:m2n_wr})}
  & \makecell{\textbf{WR} \\ \scriptsize(Thm~\ref{thm:m2n_wr})}
  & \makecell{\textbf{WR} \\ \scriptsize(Thm~\ref{thm:m2n_wr})}
  & \makecell{\textbf{WR} \\ \scriptsize(Thm~\ref{thm:m2n_wr})}
  & \makecell{\textbf{WR} \\ \scriptsize(Thm~\ref{thm:m2n_wr})}
  & \makecell{\textbf{WR} \\ \scriptsize(Thm~\ref{thm:m2n_wr})} \\
\hline
\makecell{$\mu'_{2n+1}$ \\ \scriptsize$(n\geq 1)$}
  & \makecell{\textbf{WR} \\ \scriptsize(Thm~\ref{thm:w2n_wr})}
  & \makecell{\textbf{Open} \\ \scriptsize(Prob.~\ref{prob1})}
  & \makecell{\textbf{WR} \\ \scriptsize(Cor.~\ref{cor:3col})}
  & \makecell{\textbf{NWR} \\ \scriptsize(Thm~\ref{thm:mu2n1_k4})}
  & \makecell{\textbf{WR} \\ \scriptsize(Thm~\ref{thm:m2n_wr})}
  & \makecell{\textbf{NWR} \\ \scriptsize(Thm~\ref{thm:mu2n1_mu2m1})}
  & \makecell{\textbf{WR} \\ \scriptsize(Thm~\ref{thm:m2n_wr})}
  & \makecell{\textbf{NWR} \\ \scriptsize(if $n=m$)} \\
\hline
\end{tabular}
\caption{We assume $n\leq m$. Theorem~\ref{thm:mu2n1_mu2m1} gives the cases of $\mu_{2n+1}\times\mu'_{2n+1}$ and $\mu'_{2n+1}\times\mu'_{2n+1}$.}\label{tab-results}
\end{table}

\section{Tensor product of graphs}\label{tensor-sec}

The tensor product is one of the four standard graph products, arising
naturally in the study of graph homomorphisms and graph colourings. We
recall its formal definition below.

\begin{definition}[Tensor product of graphs]
Let $G$ and $H$ be two simple undirected graphs with vertex sets
$V(G)$ and $V(H)$, and edge sets $E(G)$ and $E(H)$, respectively.
The \emph{tensor product} $G \times H$ is the graph defined as follows:
\begin{itemize}
    \item \textbf{Vertex set:}
    \[
        V(G \times H) \;=\; V(G) \times V(H)
        \;=\; \bigl\{(u,\, v) \mid u \in V(G),\; v \in V(H)\bigr\}.
    \]
    \item \textbf{Edge set:} Two vertices $(u_1, v_1)$ and $(u_2, v_2)$
    in $G \times H$ are adjacent \emph{if and only if}
    \[
        u_1 u_2 \in E(G) \quad \text{and} \quad v_1 v_2 \in E(H).
    \]
\end{itemize}
\end{definition}

\begin{remark}
The number of vertices of $G \times H$ is $|V(G)| \cdot |V(H)|$, and
the degree of any vertex $(u,v)$ in $G \times H$ equals
$\deg_G(u) \cdot \deg_H(v)$.
\end{remark}

\begin{remark}
In general, the full tensor product $G \times H$ is a much larger and
more complex graph than either $G$ or $H$ alone. For example, the full
graph $W_5 \times W_5$ has $6 \times 6 = 36$ vertices, and the degree
of each vertex $(u,v)$ depends on the types of both $u$ and $v$: a
rim--rim pair has degree $2 \times 2 = 4$, a rim--hub pair has degree
$2 \times 4 = 8$, and the hub--hub vertex has degree $4 \times 4 = 16$.
In this paper we do not study the full product directly. Instead, we
identify within it an induced subgraph that is already known to be
non-word-representable, which immediately implies that the full product
is non-word-representable.
\end{remark}

\begin{proposition}[\cite{hammack2011}]
Let $A_G$ and $A_H$ be the adjacency matrices of $G$ and $H$
respectively. Then $A_{G \times H} = A_G \times A_H$, where $\times$
denotes the Tensor product of matrices.
\end{proposition}

\begin{example}
Let $G = H = C_3$ with vertices $\{1,2,3\}$ and edges $\{12,23,13\}$.
The adjacency matrix of $C_3$ is
\[
A_{C_3} =
\begin{pmatrix}0&1&1\\1&0&1\\1&1&0\end{pmatrix}.
\]
By Proposition~2.1, $A_{C_3 \times C_3} = A_{C_3} \times A_{C_3}$.
The product $C_3 \times C_3$ has $3 \cdot 3 = 9$ vertices
$(i,j)$ for $i,j \in \{1,2,3\}$, and every vertex $(i,j)$ has degree
$\deg(i)\cdot\deg(j) = 2\cdot 2 = 4$, giving $18$ edges in total.
\end{example}

In this paper we need the following well-known results.

\begin{theorem}[\cite{hammack2011}]
\label{thm:commutativity}
For any graphs $G$ and $H$, the tensor product is commutative, that is,
\[
G \times H \cong H \times G.
\]
\end{theorem}

\begin{theorem}[\cite{hammack2011}]
\label{thm:chromatic}
For any graphs $G$ and $H$, the chromatic number of their tensor product satisfies
\[
\chi(G \times H) \leq \min\{\chi(G), \chi(H)\}.
\]
\end{theorem}

In particular, Theorem~\ref{thm:chromatic} implies that if at least one of the graphs $G$ and $H$ is $3$-colourable, then $G \times H$ is $3$-colourable.

\section{Word-representability of tensor products with wheel graphs}\label{wr-tensor-prod-wheel}

We begin with a general theorem involving even wheels $W_{2n}$.

\begin{theorem}
\label{thm:w2n_wr}
Let $G$ be any graph and
$n \geq 2$. Then $W_{2n} \times G\cong G \times W_{2n}$ is word-representable.
\end{theorem}

\begin{proof}
Since $W_{2n}$ is $3$-colourable, by Theorem~\ref{thm:chromatic} the graph $W_{2n} \times G$ is $3$-colourable, and by Corollary~\ref{cor:3col} it is word-representable.
\end{proof}

In the remaining theorems in this section, we consider tensor products of graphs containing odd wheels as induced subgraphs and show that, in each case, the resulting graph is not word-representable.

\subsection{Graphs containing odd wheels}

\begin{theorem}
\label{thm:K4xK4}
If both $G$ and $H$ contain $W_3$ then $G \times H \cong H \times G$ is non-word-representable.
\end{theorem}

\begin{proof}
We claim that any vertex in $W_3 \times W_3$, which is an induced subgraph of $G \times H$, has a non-comparability neighborhood. Hence, by Theorem~\ref{non-comp-neighb}, $W_3 \times W_3$, and therefore $G \times H$, are not word-representable. 

To complete the proof of the theorem, it suffices to justify the claim for the vertex $(1,1)$ and its neighbourhood formed by the vertices
$(2,2)$, $(2,3)$, $(2,4)$, $(3,2)$, $(3,3)$, $(3,4)$, $(4,2)$, $(4,3)$, $(4,4)$,
which is presented in Figure~\ref{neib-(1,1)}.

\begin{figure}[!t]
\begin{center}
\begin{tikzpicture}[
    scale=1,
    every node/.style={circle, draw, inner sep=1pt, minimum size=12pt},
    every edge/.style={line width=0.4pt}
]

\node (22) at (0,4) {$(2,2)$};
\node (23) at (3,4) {$(2,3)$};
\node (24) at (6,4) {$(2,4)$};

\node (32) at (0,2) {$(3,2)$};
\node (33) at (3,2) {$(3,3)$};
\node (34) at (6,2) {$(3,4)$};

\node (42) at (0,0) {$(4,2)$};
\node (43) at (3,0) {$(4,3)$};
\node (44) at (6,0) {$(4,4)$};

\draw (22)--(33);
\draw (22)--(34);
\draw (22)--(43);

\draw (23)--(32);
\draw (23)--(34);
\draw (23)--(42);
\draw (23)--(44);

\draw (24)--(32);
\draw (24)--(33);
\draw (24)--(43);

\draw (32)--(43);
\draw (32)--(44);

\draw (33)--(42);
\draw (33)--(44);

\draw (34)--(42);
\draw (34)--(43);

\draw[bend left=15] (42) to (24); 
\draw[bend right=15] (22) to (44); 

\end{tikzpicture}
\end{center}
\caption{The neighbourhood of the vertex $(1,1)$ in $W_3 \times W_3$.}
\label{neib-(1,1)}
\end{figure}
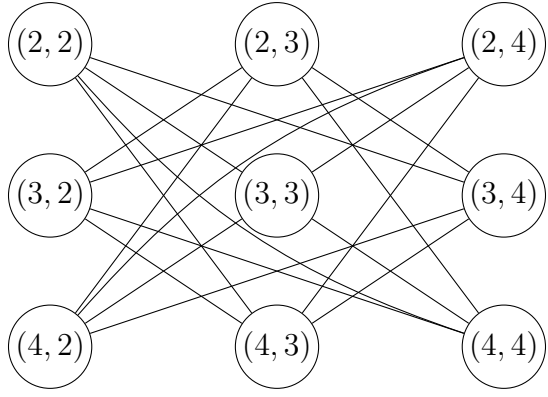

Without loss of generality, suppose that $(2,2)\rightarrow (3,3)$ is an edge in a transitive orientation. Since $(2,2)$ is adjacent to both $(3,3)$ and $(3,4)$, while $(3,3)$ is not adjacent to $(3,4)$, transitivity forces $(2,2)\rightarrow (3,4)$.
Similarly, we obtain $(2,2)\rightarrow (4,3)$.

Now, from $(2,2)\rightarrow (3,4)$, it follows that $(2,3)\rightarrow (3,4)$, which in turn implies $(2,3)\rightarrow (3,2)$. This then forces $(4,3)\rightarrow (3,2)$. However, this contradicts transitivity, since there is a directed path of length~2, $(2,2)\rightarrow(4,3)\rightarrow(
3,2)$, while the vertices $(2,2)$ and $(3,2)$ are not adjacent.
\end{proof}

\begin{theorem}
\label{thm:w3}
If $G$ contains $W_3$ and $H$ contains $W_{2m+1}$ for some $m \geq 2$,
then $G \times H \cong H \times G$ is non-word-representable.
\end{theorem}

\begin{proof}
We claim that the graph $W_3 \times W_{2m+1}$, and hence $G \times H$, contains $W_{2m+1}$ as an induced subgraph and therefore is not word-representable by Theorem~\ref{thm:w5}. Assuming that $W_5$ is labelled as in Figure~\ref{fig:wheels}, and $W_{2m+1}$ is labelled following the same pattern, we claim that the $2m+2$ vertices in $W_3 \times W_{2m+1}$ shown in Figure~\ref{fig:fig4} form an induced $W_{2m+1}$. It remains to justify this claim.

The rim follows an alternating pattern using only vertices $1, 2, 3$
from $W_3$: the rim vertices are
$(1,1),(2,2),(1,3),(2,4),(1,5),\ldots,(2,2m),(3,2m+1)$.
Each consecutive pair is adjacent in $W_3 \times W_{2m+1}$ because the
$W_3$-coordinates alternate between $1$ and $2$ (which are adjacent
in $W_3$), while the $W_{2m+1}$-coordinates are consecutive on the rim of
$W_{2m+1}$.

The vertex $(4,2m+2)$ serves as the hub because $4$ is the hub of
$W_3$ (adjacent to all of $1,2,3$) and $2m+2$ is the hub of
$W_{2m+1}$ (adjacent to all rim vertices). Hence, $(4,2m+2)$
is adjacent to every rim vertex in $W_3 \times W_{2m+1}$.

Since the $W_{2m+1}$-coordinates of non-adjacent rim pairs are non-adjacent in
$W_{2m+1}$, no chords exist, and the induced subgraph of $W_3 \times W_{2m+1}$ is exactly
$W_{2m+1}$. The proof is complete.
\end{proof}

\begin{figure}[!t]
\centering
\begin{tikzpicture}[
    V/.style={circle, draw, thick, fill=white,
              minimum size=30pt, inner sep=0pt,
              align=center, font=\footnotesize},
    scale=1
]
\node[V] (r1) at (90:2.4)   {$(1,1)$};
\node[V] (r2) at (39:2.4)   {$(2,2)$};
\node[V] (r3) at (347:2.4)  {$(1,3)$};
\node[V] (r4) at (296:2.4)  {$(2,4)$};
\node[V] (r5) at (244:2.4)  {$(1,5)$};
\node[V] (r6) at (193:2.4)  {$(2,2m)$};
\node[V] (r7) at (142:2.4)  {$(3,$\\$2m{+}1)$};
\node[V] (hub) at (0,0)     {$(4,$\\$2m{+}2)$};
\draw[thick] (r1)--(r2)--(r3)--(r4)--(r5);
\draw[thick, dotted, line width=1.2pt] (r5)--(r6);
\draw[thick] (r6)--(r7)--(r1);
\draw[thick] (hub)--(r1) (hub)--(r2) (hub)--(r3) (hub)--(r4)
             (hub)--(r5) (hub)--(r6) (hub)--(r7);
\end{tikzpicture}
\caption{An induced subgraph of the graph $W_3 \times W_{2m+1}$ isomorphic to $W_{2m+1}$.}
\label{fig:fig4}
\end{figure}
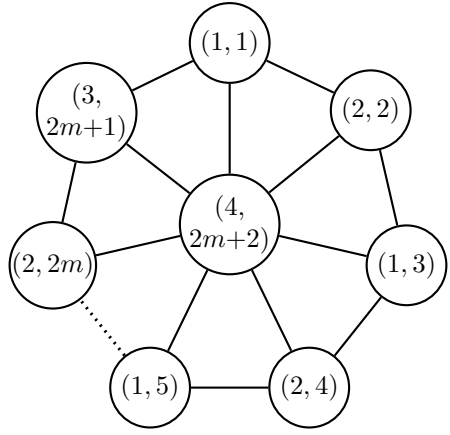

\begin{theorem}
\label{thm:wheel}
Let $G$ contain an induced subgraph $W_{2n+1}$ and let $H$ contain an
induced subgraph $W_{2m+1}$, where $n \leq m$ and $m \geq 2$.
Then $G \times H \cong H \times G$ is non-word-representable.
\end{theorem}

\begin{proof}
We claim that the graph $W_3 \times W_{2m+1}$, and hence $G \times H$, contains $W_{2m+1}$ as an induced subgraph and is therefore not word-representable by Theorem~\ref{thm:w5}. It remains to justify this claim. We consider separately the cases $n = m$ and $n < m$, where we assume that $W_{2n+1}$ and $W_{2m+1}$ are labelled according to the pattern in Figure~\ref{fig:wheels}. \\[-3mm]

\noindent
If $n = m$, consider the vertices
$(1,1),(2,2),\ldots,(2n+1,2n+1)$ as the rim and $(2n+2,2n+2)$ as the
hub, as shown in Figure~\ref{fig:fig2111}. The fact that these vertices form an induced $W_{2n+1}$ is straightforward, since both coordinates come from the same graph. \\[-3mm]

\begin{figure}[!t]
\centering
\begin{tikzpicture}[
    V/.style={circle, draw, thick, fill=white,
              minimum size=30pt, inner sep=0pt,
              align=center, font=\footnotesize},
    scale=1
]
\node[V] (r1) at (90:2.4)   {$(1,1)$};
\node[V] (r2) at (18:2.4)   {$(2,2)$};
\node[V] (r3) at (306:2.4)  {$(3,3)$};
\node[V] (r4) at (234:2.4)  {$(2n,2n)$};
\node[V] (r5) at (162:2.4)  {$(2n{+}1,$\\$2n{+}1)$};
\node[V] (hub) at (0,0)     {$(2n{+}2,$\\$2n{+}2)$};
\draw[thick] (r1)--(r2);
\draw[thick] (r2)--(r3);
\draw[thick, dotted, line width=1.2pt] (r3)--(r4);
\draw[thick] (r4)--(r5);
\draw[thick] (r5)--(r1);
\draw[thick] (hub)--(r1) (hub)--(r2) (hub)--(r3)
             (hub)--(r4) (hub)--(r5);
\end{tikzpicture}
\caption{An induced subgraph of the graph $W_{2n+1} \times W_{2n+1}$ isomorphic to $W_{2n+1}$.}
\label{fig:fig2111}
\end{figure}
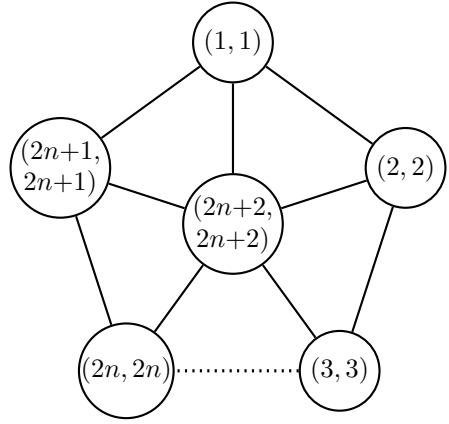

\noindent
If $n<m$, we find an induced $W_{2m+1}$ inside $W_{2n+1} \times W_{2m+1}$ as shown in Figure~\ref{fig:fig44111}. 

Indeed, the rim consists of $2m+1$ vertices:
$(1,1),(2,2),\ldots,(2n-1,2n-1),(2n,2n),$ $(2n-1,2n+1),(2n,2n+2),(2n-1,2n+3),(2n,2n+4),\ldots,(2n-1,2m-1),(2n,2m),(2n+1,2m+1)$,
where the pattern alternates between two consecutive rim vertices of
$W_{2n+1}$, namely $2n$ and $2n-1$, after the initial $2n-1$ steps. Each consecutive pair is
adjacent in $W_{2m+1} \times W_{2n+1}$ because both the $W_{2m+1}$- and $W_{2n+1}$-coordinates are
adjacent in their respective graphs.

The vertex $(2n+2, 2m+2)$ is the hub because $2n+2$ is the hub of $W_{2n+1}$ and $2m+2$ is the hub of $W_{2m+1}$. Hence, $(2n+2,2m+2)$ is adjacent to all rim vertices in $W_{2n+1} \times W_{2m+1}$.
Since no chords exist among non-adjacent rim vertex pairs in $W_{2m+1}$, the induced subgraph on the $2m+2$ vertices is exactly $W_{2m+1}$, which completes the proof.
\end{proof}

\begin{figure}[!t]
\centering

\begin{tikzpicture}[
    V/.style={circle, draw, thick, fill=white,
              minimum size=30pt, inner sep=0pt,
              align=center, font=\scriptsize},
    scale=1
]

\node[V] (r1)  at (90:3.0)   {$(1,1)$};
\node[V] (r2)  at (60:3.0)   {$(2,2)$};
\node[V] (r3)  at (30:3.0)   {$(3,3)$};
\node[V] (r4)  at (0:3.0)    {$(2n{-}1,$\\$2n{-}1)$};
\node[V] (r5)  at (330:3.0)  {$(2n,$\\$2n)$};
\node[V] (r6)  at (300:3.0)  {$(2n{-}1,$\\$2n{+}1)$};
\node[V] (r7)  at (270:3.0)  {$(2n,$\\$2n{+}2)$};
\node[V] (r8)  at (240:3.0)  {$(2n{-}1,$\\$2n{+}3)$};
\node[V] (r9)  at (210:3.0)  {$(2n,$\\$2n{+}4)$};
\node[V] (r10) at (180:3.0)  {$(2n{-}1,$\\$2m{-}1)$};
\node[V] (r11) at (150:3.0)  {$(2n,$\\$2m)$};
\node[V] (r12) at (120:3.0)  {$(2n{+}1,$\\$2m{+}1)$};

\node[V, minimum size=35pt] (hub) at (0,0) {$(2n{+}2,$\\$2m{+}2)$};

\draw[thick] (r1)--(r2)--(r3);
\draw[thick, dotted, line width=1.2pt] (r3)--(r4);
\draw[thick] (r4)--(r5)--(r6)--(r7)--(r8)--(r9);
\draw[thick, dotted, line width=1.2pt] (r9)--(r10);
\draw[thick] (r10)--(r11)--(r12)--(r1);

\draw[thick] (hub)--(r1)  (hub)--(r2)  (hub)--(r3)
             (hub)--(r4)  (hub)--(r5)  (hub)--(r6)
             (hub)--(r7)  (hub)--(r8)  (hub)--(r9)
             (hub)--(r10) (hub)--(r11) (hub)--(r12);

\end{tikzpicture}
\caption{An induced subgraph of $W_{2n+1} \times W_{2m+1}$ isomorphic to $W_{2m+1}$ for $n < m$.}
\label{fig:fig44111}
\end{figure}
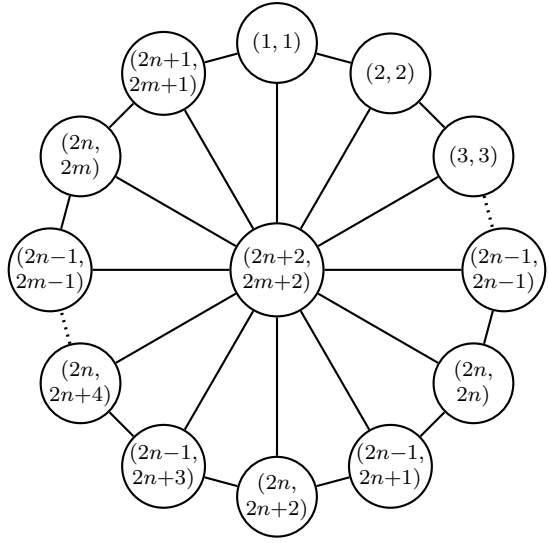

\begin{example}
In Figure~\ref{fig:fig5} we illustrate the construction in Figure~\ref{fig:fig44111} in the case $n=2$, $m=4$ and $n=3$, $m=5$.
\end{example}

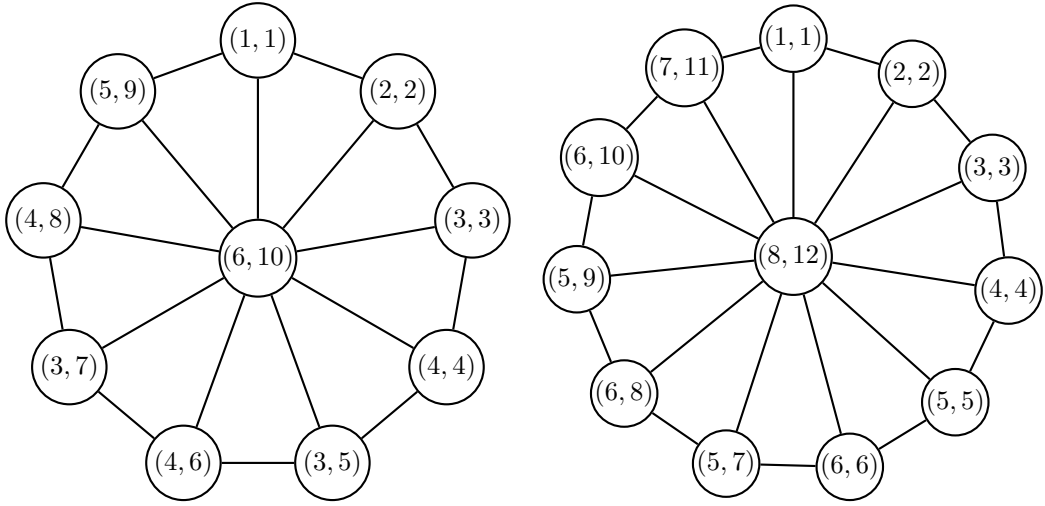
\begin{figure}

\begin{center}
\begin{tabular}{cc}
\begin{tikzpicture}[
    V/.style={circle, draw, thick, fill=white,
              minimum size=28pt, inner sep=0pt,
              align=center, font=\footnotesize},
    scale=1.2
]
\node[V] (r1)  at (90:2.4)   {$(1,1)$};
\node[V] (r2)  at (50:2.4)   {$(2,2)$};
\node[V] (r3)  at (10:2.4)   {$(3,3)$};
\node[V] (r4)  at (330:2.4)  {$(4,4)$};
\node[V] (r5)  at (290:2.4)  {$(3,5)$};
\node[V] (r6)  at (250:2.4)  {$(4,6)$};
\node[V] (r7)  at (210:2.4)  {$(3,7)$};
\node[V] (r8)  at (170:2.4)  {$(4,8)$};
\node[V] (r9)  at (130:2.4)  {$(5,9)$};
\node[V] (hub) at (0,0)      {$(6,10)$};
\draw[thick] (r1)--(r2)--(r3)--(r4)--(r5)--(r6)--(r7)--(r8)--(r9)--(r1);
\draw[thick] (hub)--(r1) (hub)--(r2) (hub)--(r3) (hub)--(r4) (hub)--(r5)
             (hub)--(r6) (hub)--(r7) (hub)--(r8) (hub)--(r9);
\end{tikzpicture}

&

\begin{tikzpicture}[
    V/.style={circle, draw, thick, fill=white,
              minimum size=25pt, inner sep=0pt,
              align=center, font=\footnotesize},
    scale=1.2
]
\node[V] (r1)  at (90:2.4)   {$(1,1)$};
\node[V] (r2)  at (57:2.4)   {$(2,2)$};
\node[V] (r3)  at (24:2.4)   {$(3,3)$};
\node[V] (r4)  at (351:2.4)  {$(4,4)$};
\node[V] (r5)  at (318:2.4)  {$(5,5)$};
\node[V] (r6)  at (285:2.4)  {$(6,6)$};
\node[V] (r7)  at (252:2.4)  {$(5,7)$};
\node[V] (r8)  at (219:2.4)  {$(6,8)$};
\node[V] (r9)  at (186:2.4)  {$(5,9)$};
\node[V] (r10) at (153:2.4)  {$(6,10)$};
\node[V] (r11) at (120:2.4)  {$(7,11)$};
\node[V] (hub) at (0,0)      {$(8,12)$};
\draw[thick] (r1)--(r2)--(r3)--(r4)--(r5)--(r6)--(r7)--(r8)--(r9)--(r10)--(r11)--(r1);
\draw[thick] (hub)--(r1)  (hub)--(r2)  (hub)--(r3)  (hub)--(r4)  (hub)--(r5)
             (hub)--(r6)  (hub)--(r7)  (hub)--(r8)  (hub)--(r9)  (hub)--(r10)
             (hub)--(r11);
\end{tikzpicture}

\end{tabular}

\end{center}

\caption{An induced subgraph $W_{9}$ of $W_{5}\times W_{9}$ (left) and an induced subgraph $W_{11}$ of $W_{7}\times W_{11}$ (right).}
\label{fig:fig5}
\end{figure}

\subsection{A graph containing an odd wheel and a graph containing $K_n$, $\mu_n$, or $\mu'_n$}

Before the following theorem, note that $G \times K_n$ is word-representable for any graph $G$ and $n \leq 3$, since $G \times K_n$ is then 3-colourable and Corollary~\ref{cor:3col} can be applied.

\begin{theorem}
\label{thm:w3-1}
If $G$ contains $W_{2m+1}$ for some $m \geq 1$ and $H$ contains $K_n$ for $n \geq 4$,
then $G \times H \cong H \times G$ is non-word-representable.
\end{theorem}

\begin{proof}
The graph $H$ contains $K_4 \cong W_3$, and Theorems~\ref{thm:K4xK4} and~\ref{thm:w3} yield the result.
\end{proof}

\begin{theorem}
\label{thm:mu2n1_k4}
Let $G$ contain $\mu_{2n+1}$ or $\mu'_{2n+1}$ for $n \geq 1$ and let $H$ contain
$W_3$. Then $G \times H \cong H\times G$ is non-word-representable.
\end{theorem}

\begin{proof}
We show that $\mu_{2n+1} \times W_3$ contains $\mu_{2n+1}$ as an induced subgraph, and hence so does  $G \times H$. Therefore $G \times H$ is non-word-representable by Remark~\ref{odd-M-non-wr}. The argument for $\mu'_{2n+1} \times W_3$ is identical, and therefore is omitted.

The graph $\mu_{2n+1}$ has $4n+3$ vertices arranged in three layers: the top row $1,2,\ldots,2n+1$ forming the cycle $C_{2n+1}$, the shadow row $2n+2,2n+3,\ldots,4n+2$, and the root vertex $4n+3$. The graph $W_3$ has vertex set $\{1,2,3,4\}$ with all edges present. The induced copy of $\mu_{2n+1}$ in $\mu_{2n+1} \times W_3$ consists of the following vertices.

The top row consists of
\[
(1,1),\ (2,2),\ (3,1),\ (4,2),\ \ldots,\ (2n-1,1),\ (2n,2),\ (2n+1,3),
\]
the shadow row consists of
\[
(2n+2,4),\ (2n+3,4),\ (2n+4,4),\ \ldots,\ (4n+2,4),
\]
and the root vertex is $(4n+3,1)$.

It is straightforward to verify that the subgraph induced by these $4n+3$ vertices is isomorphic to $\mu_{2n+1}$, completing the proof.
\end{proof}
\begin{example}
In Figure~\ref{fig:fig12}, we illustrate the construction of the non-word-representable induced subgraph in the proof of Theorem~\ref{thm:mu2n1_k4} for the case $n=2$.
\end{example}

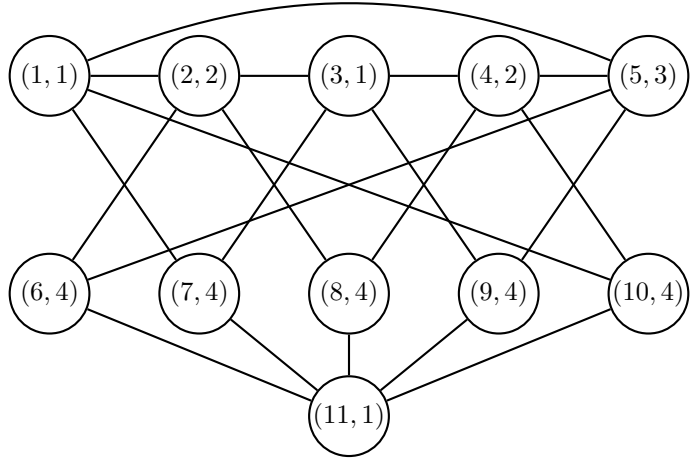
\begin{figure}[!t]
\centering

\begin{tikzpicture}[
    V/.style={circle, draw, thick, fill=white,
              minimum size=30pt, inner sep=0pt, font=\footnotesize},
    scale=0.9
]
\node[V] (r11) at (0,   0)  {$(1,1)$};
\node[V] (r22) at (2.2, 0)  {$(2,2)$};
\node[V] (r31) at (4.4, 0)  {$(3,1)$};
\node[V] (r42) at (6.6, 0)  {$(4,2)$};
\node[V] (r53) at (8.8, 0)  {$(5,3)$};
\node[V] (s64)  at (0,  -3.2){$(6,4)$};
\node[V] (s74)  at (2.2,-3.2){$(7,4)$};
\node[V] (s84)  at (4.4,-3.2){$(8,4)$};
\node[V] (s94)  at (6.6,-3.2){$(9,4)$};
\node[V] (s104) at (8.8,-3.2){$(10,4)$};
\node[V] (root) at (4.4,-5.0){$(11,1)$};

\draw[thick] (r11)--(r22)--(r31)--(r42)--(r53);
\draw[thick, bend left=22] (r11) to (r53);

\draw[thick] (s64)--(r22)  (s64)--(r53);
\draw[thick] (s74)--(r11)  (s74)--(r31);
\draw[thick] (s84)--(r22)  (s84)--(r42);
\draw[thick] (s94)--(r31)  (s94)--(r53);
\draw[thick] (s104)--(r42) (s104)--(r11);

\draw[thick] (root)--(s64) (root)--(s74) (root)--(s84)
             (root)--(s94) (root)--(s104);

\end{tikzpicture}
\caption{An induced subgraph $\mu_5$ in $\mu_5 \times W_3$.}
\label{fig:fig12}
\end{figure}

\section{Remaining cases in our studies}\label{remaining-sec}

\subsection{Word-representability of tensor products with  complete graphs}

\begin{theorem}
\label{thm:kn_km}
Let $G = K_n \times K_m$. Then $G$ is word-representable if and only
if $\min\{n,m\} \leq 3$.
\end{theorem}

\begin{proof}
If $\min\{n,m\} \leq 3$, then $G$ is 3-colourable and hence word-representable by Corollary~\ref{cor:3col}.

Conversely, if $\min\{n,m\} > 3$, then $G$ contains $K_4 \times K_4 \cong W_3 \times W_3$, and the result follows from Theorem~\ref{thm:K4xK4}.
\end{proof}

\begin{corollary}
\label{cor-Kn-Km}
If $G$ and $H$ contain $K_n$ and $K_m$, respectively, for $n,m \geq 4$, then $G \times H \cong H \times G$ is non-word-representable.
\end{corollary}

\subsection{Word-representability of tensor products with (extended) Mycielski of cycle graphs}

\begin{theorem}
\label{thm:m2n_wr}
Let $G$ be any graph. Then $\mu_{2n}\times G\cong G \times \mu_{2n}$ ($\mu'_{2n}\times G\cong G \times \mu'_{2n}$) is word-representable for any $n \geq 2$. 
\end{theorem}

\begin{proof}
Since $\mu_{2n}$ (resp., $\mu'_{2n}$) is $3$-colourable, by Theorem~\ref{thm:chromatic} the graph $\mu_{2n} \times G$ (resp.,  $\mu'_{2n} \times G$) is $3$-colourable, and by Corollary~\ref{cor:3col} it is word-representable.
\end{proof}

\begin{theorem}
\label{thm:mu2n1_mu2m1}
If $G$ contains $\mu_{2n+1}$ or $\mu'_{2n+1}$ and $H$ contains $\mu_{2m+1}$ where $1\leq n \leq m$.
Then $G \times H \cong H \times G$ is non-word-representable.
\end{theorem}

\begin{proof}
We assume that $G$ contains $\mu_{2n+1}$; all our arguments are identical for the case when $G$ contains $\mu'_{2n+1}$. It is sufficient to prove that $\mu_{2n+1} \times \mu_{2n+1}$ is not word-representable.

If $n = m$, then $\mu_{2n+1} \times \mu_{2n+1}$ clearly contains $\mu_{2n+1}$ as an induced subgraph on the vertices $(1,1), (2,2), \ldots, (4n+3,4n+3)$, which is non-word-representable. Hence, $G \times H$ is not word-representable in this case.

Now assume $n < m$. Then $\mu_{2n+1} \times \mu_{2m+1}$ contains $\mu_{2m+1}$ as an induced subgraph given by the following vertices.

The top row of this subgraph consists of the vertices
\begin{align*}
& (1,1),\ (2,2),\ \ldots,\ (2n+1,2n+1),\ (2n,2n+2),\ (2n+1,2n+3),\ (2n,2n+4),\ \ldots, \\
& (2n+1,2m-1),\ (2n,2m),\ (2n+1,2m+1),
\end{align*}
the shadow row consists of
\begin{align*}
& (2n+2,2m+2),\ (2n+3,2m+3),\ \ldots,\ (4n+1,4m+1),\ (4n,4m+2),\ (4n+1,4m+3), \\
& (4n,4m+1),\ \ldots,\ (4n,4m),\ (4n+1,4m+1),\ (4n+2,4m+2),
\end{align*}
and the root vertex is $(4n+3,4m+3)$.

We verify that the induced subgraph on these vertices is isomorphic to $\mu_{2m+1}$.

First, the vertices in the top row form a cycle of length $2m+1$. Indeed, by construction, consecutive vertices are connected in both coordinates and hence are adjacent in $\mu_{2n+1} \times \mu_{2m+1}$, while non-consecutive vertices fail this condition and are not adjacent. Thus, the top row induces $C_{2m+1}$.

Second, each vertex in the shadow row corresponds uniquely to a vertex in the top row, and adjacency between the top row and the shadow row follows the same rule as in the Mycielskian construction: a shadow vertex is adjacent precisely to the neighbours of its corresponding top vertex. This follows directly from the definition of the tensor product.

Third, the root vertex $(4n+3,4m+3)$ is adjacent to all vertices in the shadow row and to no vertex in the top row, again matching the definition of the Mycielskian.

Finally, no additional edges are present among these vertices, since any pair not prescribed above fails the adjacency condition in the second coordinate. Hence, the subgraph is induced. Therefore, the induced subgraph is isomorphic to $\mu_{2m+1}$, which completes the proof.
\end{proof}

\begin{figure}[!t]
\centering

\begin{tikzpicture}[
    V/.style={circle, draw, thick, fill=white,
              minimum size=24pt, inner sep=0pt, font=\footnotesize},
    scale=1.0
]
\node[V] (r11) at (0,   0)   {$(1,1)$};
\node[V] (r22) at (1.8, 0)   {$(2,2)$};
\node[V] (r33) at (3.6, 0)   {$(3,3)$};
\node[V] (r24) at (5.4, 0)   {$(2,4)$};
\node[V] (r35) at (7.2, 0)   {$(3,5)$};
\node[V] (r26) at (9.0, 0)   {$(2,6)$};
\node[V] (r37) at (10.8,0)   {$(3,7)$};
\node[V] (s48)  at (0,  -3.0) {$(4,8)$};
\node[V] (s59)  at (1.8,-3.0) {$(5,9)$};
\node[V] (s610) at (3.6,-3.0) {$(6,10)$};
\node[V] (s511) at (5.4,-3.0) {$(5,11)$};
\node[V] (s612) at (7.2,-3.0) {$(6,12)$};
\node[V] (s513) at (9.0,-3.0) {$(5,13)$};
\node[V] (s614) at (10.8,-3.0){$(6,14)$};
\node[V] (root) at (5.4,-5) {$(7,15)$};

\draw[thick] (r11)--(r22)--(r33)--(r24)--(r35)--(r26)--(r37);
\draw[thick, bend left=18] (r11) to (r37);

\draw[thick] (s48)--(r22)  (s48)--(r37);
\draw[thick] (s59)--(r11)  (s59)--(r33);
\draw[thick] (s610)--(r22) (s610)--(r24);
\draw[thick] (s511)--(r33) (s511)--(r35);
\draw[thick] (s612)--(r24) (s612)--(r26);
\draw[thick] (s513)--(r35) (s513)--(r37);
\draw[thick] (s614)--(r26) (s614)--(r11);

\draw[thick] (root)--(s48)  (root)--(s59)  (root)--(s610)
             (root)--(s511) (root)--(s612) (root)--(s513)
             (root)--(s614);

\end{tikzpicture}
\caption{An induced subgraph $\mu_7$ in $\mu_3 \times \mu_7$.}
\label{fig:fig10}
\end{figure}
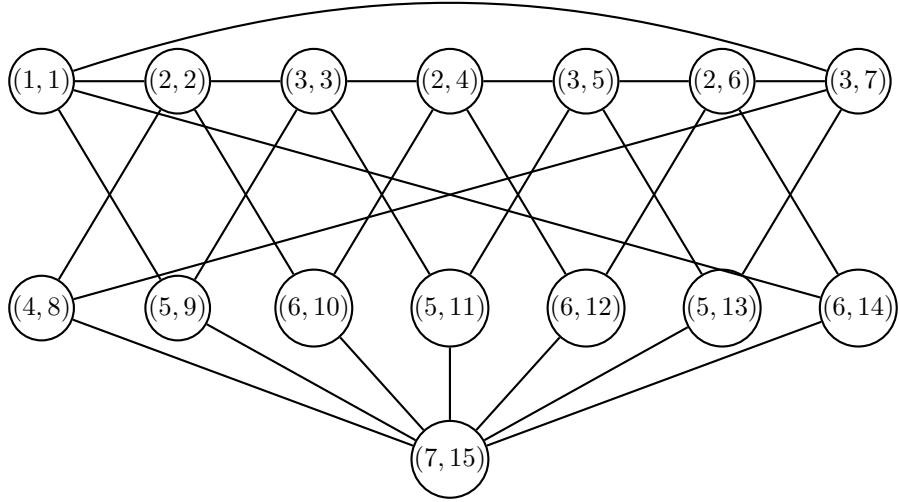


\begin{example}
In Figures~\ref{fig:fig10} and~\ref{fig:fig11}, we illustrate Theorem~\ref{thm:mu2n1_mu2m1} for $n = 1$ and
$m = 3$ (that is, $\mu_3 \times \mu_7$), and for $n = 2$ and
$m = 3$ (that is, $\mu_5 \times \mu_7$).
\end{example}

\begin{remark}
Note that our construction in the proof of Theorem~\ref{thm:mu2n1_mu2m1} does not work for the graph $\mu'_{5}\times\mu'_{7}$, and more generally for  $\mu'_{2n+1}\times\mu'_{2m+1}$. Indeed, referring to Figure~\ref{fig:fig11}, note that there is an edge between $(6,8)$ and $(3,3)$ in  $\mu'_{5}\times\mu'_{7}$ while there is no edge between $(3,3)$ and $(8,12)$ in $\mu'_{5}\times\mu'_{7}$. Therefore, the induced subgraph formed by the vertices in Figure~\ref{fig:fig11} in $\mu'_{5}\times\mu'_{7}$ cannot be $\mu_{7}$ or $\mu'_{7}$.  Similarly, our construction in the proof of Theorem~\ref{thm:mu2n1_mu2m1} does not work for the graph $\mu_{2n+1}\times\mu'_{2m+1}$ for $n<m$.
\end{remark}

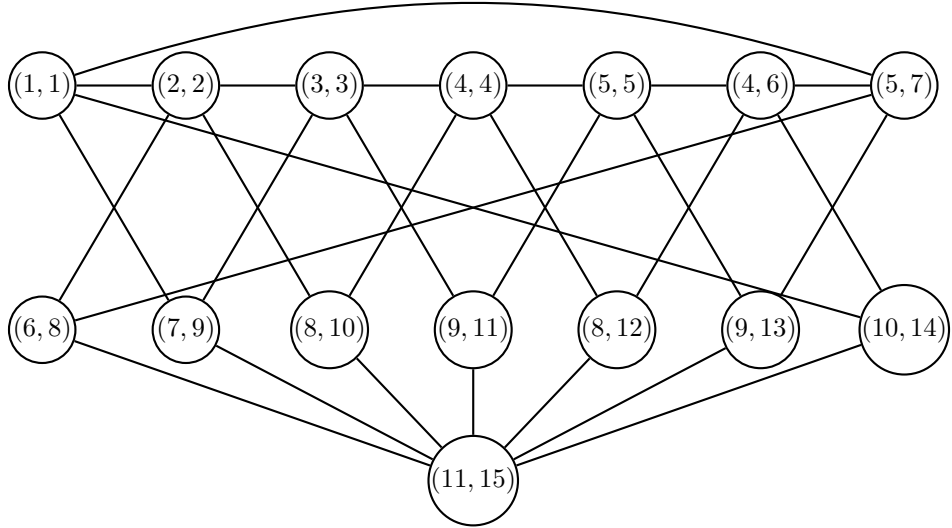
\begin{figure}[!t]
\centering

\begin{tikzpicture}[
    V/.style={circle, draw, thick, fill=white,
              minimum size=25pt, inner sep=0pt, font=\footnotesize},
    scale=0.95
]
\node[V] (r11) at (0,   0)  {$(1,1)$};
\node[V] (r22) at (2.0, 0)  {$(2,2)$};
\node[V] (r33) at (4.0, 0)  {$(3,3)$};
\node[V] (r44) at (6.0, 0)  {$(4,4)$};
\node[V] (r55) at (8.0, 0)  {$(5,5)$};
\node[V] (r46) at (10.0,0)  {$(4,6)$};
\node[V] (r57) at (12.0,0)  {$(5,7)$};
\node[V] (s68)  at (0,  -3.4){$(6,8)$};
\node[V] (s79)  at (2.0,-3.4){$(7,9)$};
\node[V] (s810) at (4.0,-3.4){$(8,10)$};
\node[V] (s911) at (6.0,-3.4){$(9,11)$};
\node[V] (s812) at (8.0,-3.4){$(8,12)$};
\node[V] (s913) at (10.0,-3.4){$(9,13)$};
\node[V] (s1014) at (12.0,-3.4){$(10,14)$};
\node[V] (root) at (6.0,-5.5){$(11,15)$};

\draw[thick] (r11)--(r22)--(r33)--(r44)--(r55)--(r46)--(r57);
\draw[thick, bend left=18] (r11) to (r57);

\draw[thick] (s68)--(r22)  (s68)--(r57);
\draw[thick] (s79)--(r11)  (s79)--(r33);
\draw[thick] (s810)--(r22) (s810)--(r44);
\draw[thick] (s911)--(r33) (s911)--(r55);
\draw[thick] (s812)--(r44) (s812)--(r46);
\draw[thick] (s913)--(r55) (s913)--(r57);
\draw[thick] (s1014)--(r46) (s1014)--(r11);

\draw[thick] (root)--(s68)   (root)--(s79)   (root)--(s810)
             (root)--(s911)  (root)--(s812)  (root)--(s913)
             (root)--(s1014);

\end{tikzpicture}
\caption{An induced subgraph $\mu_7$  in $\mu_5 \times \mu_7$.}
\label{fig:fig11}
\end{figure}


\section{Conclusion}\label{conclusion-sec}

In this paper, we initiated a systematic study of the word-representability of tensor products of graphs. Our results cover four fundamental graph families. However, our classification does not address the following case, which we state as an open problem. (Recall that in Theorem~\ref{thm:mu2n1_k4} we proved that $\mu_{2n+1} \times W_3$ is non-word-representable for all $n \geq 1$.)

\begin{problem}\label{prob1}
Is it true that $\mu_{2n+1} \times W_{2m+1}$ and $\mu'_{2n+1} \times W_{2m+1}$ is always non-word-representable for $n \geq 1$ and $m \geq 2$?
\end{problem}

\begin{problem}\label{prob2}
Is it true that $\mu_{2n+1} \times \mu'_{2m+1} $  is always non-word-representable for $n<m$?
\end{problem}

\begin{problem}\label{prob3}
Is it true that $\mu'_{2n+1} \times \mu'_{2m+1} $  is always non-word-representable for $n<m$?
\end{problem}

We conclude with the following open problem, which is the only remaining unsolved question from Problem~7.2.5 in \cite{kitaev2025}.

\begin{problem}
Is it possible for the tensor product of two non-word-representable graphs to be word-representable?
\end{problem}


\end{document}